\documentclass{scrartcl}
\usepackage[latin9]{inputenc}
\usepackage{a4}
\usepackage{amsfonts}
\usepackage{amssymb}
\usepackage{amsmath}
\usepackage{amsthm}
\usepackage{makeidx}
\usepackage{pdflscape}
\usepackage{setspace}
\usepackage{multirow}
\usepackage{tabularx}
\usepackage{float}
\makeindex
\usepackage{nomencl}
\usepackage{geometry}
\usepackage{extarrows}
\usepackage{changepage}
\usepackage{chngcntr} 
\usepackage[pdftex]{graphicx}
\usepackage{tikz}
\usepackage[toc,page]{appendix}
\providecommand{\keywords}[1]{\textbf{\textit{Keywords: }} #1}
\theoremstyle{plain}
\newtheorem{satz}{Theorem}\numberwithin{satz}{section} 
\newtheorem{lemma}[satz]{Lemma}\numberwithin{satz}{section} 
\newtheorem{prop}[satz]{Proposition}\numberwithin{satz}{section} 
\newtheorem{kor}[satz]{Corollary}\numberwithin{satz}{section} 
\theoremstyle{definition}
\newtheorem{defi}{Definition}\numberwithin{defi}{section} 
\theoremstyle{remark}
\numberwithin{conj}{section}
\numberwithin{remark}{section}
\newcommand{\nn}{{\mathbb{N}}}   
\newcommand{\qq}{{\mathbb{Q}}}   
\newcommand{\rr}{{\mathbb{R}}}   
\newcommand{\zz}{{\mathbb{Z}}}   
\begin{document}
\title{On intersective polynomials with non-solvable Galois group}
\author{Joachim K\"onig\thanks{Department of Mathematics Education, Korea National University of Education, Cheongju 28173, South Korea. email: jkoenig@knue.ac.kr}}
\date{}
\maketitle
\begin{abstract}{We present new theoretical results on the existence of intersective polynomials (that is, integer polynomials with roots in all $\qq_p$, but not in $\qq$)
with certain prescribed Galois groups, namely the projective and affine linear groups $PGL_2(\ell)$ and $AGL_2(\ell)$ as well
as the affine symplectic groups $AGSp_4(\ell):=(\mathbb{F}_\ell)^4 \rtimes GSp_4(\ell)$. For further families of affine groups, existence results are proven conditional on the existence on certain tamely ramified Galois extensions.
We also compute explicit families of intersective polynomials for certain non-solvable groups.}
\end{abstract}
\keywords{Galois theory; polynomials; number fields; group theory; Galois representations.\\}
{\textup{2000} \textit{Mathematics Subject Classification}: \textup{11R32, 11C08}}
\section{Introduction}
Let $f\in \zz[X]$ be a monic polynomial with integer coefficients. $f$ is called {\textit{intersective}}, if $f$ has a root modulo every $n\in \nn$. Equivalently, $f$ has a root in $\qq_p$ for every prime $p$.
$f$ is called non-trivially intersective, if $f$ does not have a rational root.

A simple example of a non-trivially intersective polynomial is $(X^2-2)(X^2+7)(X^2+14)$.

In the following, ``intersective" should always be understood as ``non-trivially intersective". The existence of intersective polynomials with a given number of irreducible factors and a given Galois group $G$ is linked closely to
the following property of the group $G$.

\begin{defi}[$k$-coverable group]
Let $G$ be a finite group and $k\in \nn$. $G$ is called $k$-coverable if there exist proper subgroups $U_1,..., U_k$ of $G$ such that every element of $G$ is contained in a conjugate of one of the $U_i$, in other words if
$$G=\bigcup_{i=1}^k \bigcup_{x\in G} U_i^x.$$
\end{defi}

An easy exercise in undergraduate algebra shows that there are no $1$-coverable finite groups (see e.g.\ \cite[Lemma 13.3.2]{FJ}), and therefore, by the following exposition, there are no irreducible intersective polynomials). 
Several authors have studied $2$-coverable groups (cf.\ e.g.\ \cite{Brandl}, \cite{BL} or \cite{RS}).

Let $f\in \zz[X]$ be a monic separable polynomial with $k$ irreducible factors. It is easy to verify that, if $f$ is intersective with Galois group $G$, then $G$ must be $k$-coverable.
More precisely, let $f=f_1\cdots f_k$ be a complete factorization of the intersective polynomial $f$ over $\zz$. Let $E|\qq$ be the splitting field of $f$, and let $p$ be a prime that does not ramify in $E$.
Then the Galois group $G_p$ of the localization $E_p|\qq_p$ is cyclic and fixes one of the roots of some $f_i$, since by definition, $f$ has a root in $\qq_p$. This means that $G_p=:\langle \sigma\rangle$ is contained in 
a conjugate of the proper subgroup $U_i:=Gal(E|\qq(\alpha))$, for a root $\alpha$ of $f_i$. On the other hand, by the Chebotarev density theorem, elements of every conjugacy class of $G$ occur as generators of such a decomposition group $G_p$.
Therefore $G$ is the union of the conjugates of all $U_i$ ($i=1,...,k$).

This means that, on the other hand, a given Galois extension $E|\qq$ with a $k$-coverable Galois group $G$ will yield intersective polynomials with $k$ irreducible factors, provided that the decomposition groups at the {\textit{ramified}} primes in $E$
behave appropriately.
More precisely, Sonn shows the following in Proposition 1 of \cite{Sonn08}:
\begin{prop}
\label{sonn}
Let $K|\qq$ be a Galois extension with Galois group $G$. Then the following are equivalent:
\begin{itemize}
 \item[i)] $K$ is the splitting field of a separable and non-trivially intersective polynomial $f\in \zz[X]$ with exactly $k$ irreducible factors.
 \item[ii)] $G$ is $k$-coverable by subgroups $U_1,...,U_k$, the intersection of all conjugates $U_i^g$ (with $i=1,...,k$ and $g\in G$) is trivial,
 and every decomposition group $D_p$ in $K$ (for $p\in \mathbb{P}$) is contained in a conjugate of one of the $U_i$.
\end{itemize}
\end{prop}

Note that the condition of trivial intersection in ii) of Prop.\ \ref{sonn} is automatically satisfied if $G$ is simple, or if one of the $U_i$ is a point stabilizer in a faithful permutation action of $G$.
As this will always be the case in our examples in the following sections, we therefore will not mention this condition there.

In \cite{Sonn09}, Sonn shows that intersective polynomials exist for all non-solvable groups that occur as Galois groups over $\qq$. He also shows in \cite[Theorem 2]{Sonn08} that for {\textit{solvable}} groups that admit a $k$-covering of subgroups
with the trivial intersection property of Prop.\ \ref{sonn}ii),
intersective polynomials with exactly $k$ irreducible factors exist.\footnote{As noted in \cite[Theorem 2]{Sonn08}, such a $k$-covering is possible (for some $k$) for all non-cyclic groups.}
For most non-solvable groups (even the ones known to occur as Galois groups over $\qq$), this problem remains open:

{\textbf{Question 1}}: Let $G$ be a finite group, $k$-coverable by subgroups $U_1,...,U_k$ such that the intersection of all conjugates $U_i^g$ (with $i=1,...,k$ and $g\in G$) is trivial.
Does there exist an intersective polynomial $f=f_1\cdots f_k \in \zz[X]$ with exactly $k$ irreducible factors and with Galois group $G$?\\

In this paper, we focus on $2$-coverable groups. 
Intersective polynomials have been given for the $2$-coverable alternating and symmetric groups (\cite{RS}) as well as for the dihedral group $D_5$ (\cite{LSY}), and their existence has been shown for non-solvable Frobenius groups (\cite{Sonn08}).
Using results from the theory of Galois representations, we will derive the existence of intersective polynomials with two irreducible factors for several infinite families of groups, namely the groups $PGL_2(\ell)$ and $AGL_2(\ell)$ 
as well as the affine symplectic groups $AGSp_4(\ell)$ with an arbitrary prime $\ell$.

For small $2$-coverable groups, one can nowadays easily check the databases of number fields (\cite{JR} or \cite{KM}) to find fields with suitable ramification conditions,
yielding examples of intersective polynomials with two irreducible factors for the prescribed group. 
A somewhat more challenging task is to find nice parametric families for a given group. We give examples of such families for several non-solvable groups in Section \ref{pols}.

{\textbf{Acknowledgments:}}
I would like to thank Jack Sonn for numerous helpful discussions. I also thank the referee for several useful suggestions, including in particular an improvement of the proof of Lemma \ref{metacyc}.\\
This work was partially supported by the Israel Science Foundation (grant no.\ 577/15).

\section{Existence results}
\label{exist}

We briefly recall some basic facts about decomposition groups and inertia groups: 
\begin{lemma}
Let $E|\qq$ be a Galois extension, $p$ be a prime and $E_p|\qq_p$ be the localization at $p$. Let $D_p$, $I_p$ and $I_{p,w}$ denote the decomposition group, the inertia group and the wild inertia group at a prime ideal of $E$ extending $p$.
Then $I_{p,w}\trianglelefteq I_p \trianglelefteq D_p$ is a chain of normal subgroups of $D_p$, the quotient $D_p/I_p$ is cyclic, the quotient $I_p/I_{p,w}$ is cyclic of order coprime to $p$, and $I_{p,w}$ is a $p$-group.
In particular $D_p$ is solvable.
\end{lemma}

\subsection{The groups $PGL_2(\ell)$}
\label{pgl}
The groups $PGL_2(\ell)$, with $\ell$ a prime, are all known to occur as Galois groups over $\qq$, thanks to the theory of ordinary Galois representations.
In this section we show how such representations can be used to obtain intersective polynomials with two irreducible factors in $\zz[X]$, and with Galois group $PGL_2(\ell)$.
In \cite{BL}, it is shown that for any prime power $q$, $PGL_n(q)$ is $2$-coverable if and only if $2\le n\le 4$.\\
In particular, for $n=2$, one has the following.
\begin{satz}
\label{pgl2cov}
Let $q\ge 4$ 
be a prime power. Let $U_1$ be a point stabilizer of $PGL_2(q)$ in its natural degree $q+1$ action on $1$-dimensional subspaces. Let $U_2$ be the normalizer of a Singer cycle (i.e., of a cyclic subgroup of order $q+1$) in $PGL_2(q)$.
Then $U_1$ and $U_2$ yield a $2$-covering of $PGL_2(q)$.
\end{satz}

From now on, let $\ell \ge 5$ be a prime. 
To ensure the existence of a pair $(f,g)$ with $f,g\in \zz[X]$ monic and irreducible, such that $Gal(f\cdot g) \cong PGL_2(\ell)$ and $f\cdot g$ has a root in every $\qq_p$,
we need a $PGL_2(\ell)$-extension $K|\qq$ such that all decomposition groups $Gal(K_p|\qq_p)$ are contained either in a point stabilizer or in the normalizer of a singer cycle. 
One way to do this is to make use of results on tamely ramified extensions, as presented in \cite{AdR}.
We briefly summarize the construction given there. Cf.\ especially the proof of \cite[Thm.\ 2.1]{AdR}.

For every prime $\ell\ge 5$, there exists a semistable elliptic curve with good supersingular reduction at $\ell$, 
such that the attached Galois representation $\rho: Gal_\qq \to GL_2(\mathbb{F}_\ell)$ is surjective (the surjectivity follows from the semistability for all $\ell\ge 11$ by a theorem of Mazur (\cite[Theorem 4]{Mazur}); and the smaller $\ell$ are treated explicitly in \cite{AdR}).
Composition of this representation with the projection $GL_2(\mathbb{F}_\ell)\to PGL_2(\mathbb{F}_\ell)$ yields a Galois extension of $\qq$ with group $PGL_2(\mathbb{F}_\ell)$.

It is shown in \cite{AdR}, using the Neron-Ogg-Shafarevich criterion and Tate curves, that semistability implies that the inertia group at every prime $p\ne \ell$ is either trivial or of order $\ell$. In particular, all decomposition groups at the {\textit{ramified}} primes $p\ne \ell$ are contained in the normalizer of an $\ell$-Sylow
subgroup of $PGL_2(\ell)$, i.e.\ a conjugate of $U_1$.

Furthermore, because of good supersingular reduction at $\ell$, the inertia subgroups at $\ell$ are cyclic of order $\ell+1$ (i.e., Singer cycles), and therefore the decomposition subgroups are contained in a conjugate of $U_2$. This follows
from Proposition 12 in \cite{Serre}.

We therefore obtain:
\begin{satz}
For every prime $\ell \ge 5$, there is an intersective polynomial with two monic irreducible factors $f,g\in \zz[X]$ such that $Gal(fg|\qq) = PGL_2(\ell)$.
\end{satz}

\subsection{The groups $AGL_2(\ell)$}
\label{agl2}
For $\ell\ge 5$ a prime, we will deduce the existence of intersective polynomials with two irreducible factors for the affine linear groups $AGL_2(\ell)$ from the corresponding results for $PGL_2(\ell)$.
First note that the existence of Galois extensions over $\qq$ with Galois group $AGL_2(\ell)$ follows immediately from the existence of those with group $GL_2(\ell)$, 
since any split embedding problem with abelian kernel (here, with kernel $(C_\ell)^2$)
possesses a solution over $\qq$ (see Theorem IV.2.4 in \cite{MM}).

The groups $AGL_n(q)$ are $2$-coverable even for arbitrary $n\ge 2$ and prime powers $q$.
\begin{satz}
\label{agl2cov}
Let $q$ 
be a prime power, $n\ge 2$ and let $G:=AGL_n(q)$ be the affine linear group in its natural primitive action on $V:=(\mathbb{F}_q)^n$. 
Let $U_1$ be a point stabilizer and $U_2:=V\rtimes H$, where $H \le GL_n(q)$ is a stabilizer of a one-dimensional subspace in $V$.
Then $U_1$ and $U_2$ yield a $2$-covering of $AGL_n(q)$.
\end{satz}
\begin{proof}
Any element of $G$ is a map of the form $\varphi:x\mapsto Ax+v$, with $v\in V$, $A\in GL_n(q)$. Now if $v$ is in the image of $I_n-A$ ($I_n$ the identity matrix), then $\varphi$ fixes a point in $V$,
and therefore (by the transitivity of the action on $V$) lies in a conjugate of $U_1$.
But otherwise, $A$ must have an eigenvalue $1$, so it fixes a one-dimensional subspace and therefore lies in a conjugate of $H$. Therefore $\varphi$ lies in a conjugate of $V\rtimes H$.
\end{proof}

Now let $n=2$ and $q=\ell\ge 5$ be a prime. 

First, assume $\ell \ge 11$.
Let $E$ be a semistable elliptic curve over $\qq$ with ordinary good reduction at $\ell$. The theory of Galois representations
yields a Galois extension $K|\qq$ with group $GL_2(\ell)$ and the following ramification (cf.\ the corollary after Proposition 11 in \cite{Serre}):
\begin{itemize}
 \item[i)] At prime ideals of $K$ extending $\ell$, the inertia group either has order $\ell-1$ and is conjugate to the group of the form $\begin{pmatrix}\star &0\\ 0& 1\end{pmatrix}$, or
 has order $\ell\cdot (\ell-1)$ and is conjugate to the group of the form $\begin{pmatrix}\star &\star\\ 0& 1\end{pmatrix}$.
 \item[ii)] At prime ideals extending a prime $p\ne \ell$, by semistability, the inertia group is either trivial or of order $\ell$, conjugate to a subgroup of $\begin{pmatrix}1 &\star\\ 0& 1\end{pmatrix}$.
\end{itemize}
The normalizer of the first group in case i) is just the group of diagonal matrices in $GL_2(\ell)$, whereas the normalizers of the other cases contain a normal subgroup of order $\ell$.
This shows immediately that the decomposition group at any ramified prime is contained in a conjugate of the Borel subgroup (as these are the normalizers of an order-$\ell$ subgroup), which is just the stabilizer of a one-dimensional subspace in 
$V:=(\mathbb{F}_\ell)^2$. By \cite[Theorem IV.2.4]{MM}, there is an extension $K_2\supset K$ such that $K_2|\qq$ is Galois with group $AGL_2(\ell)$. The decomposition groups at all ramified primes in $K_2|\qq$ that extend primes already
ramified in $K|\qq$ are then contained in a conjugate of $V\rtimes H$, with $H$ the Borel subgroup in $GL_2(\ell)$. But for ramified primes of $K_2|\qq$ extending {\textit{unramified}} primes of $K|\qq$ (note that in particular those primes must extend
a rational prime $p\ne \ell$), the inertia subgroup must be cyclic of order $\ell$ and therefore the decomposition subgroup normalizes a 1-dimensional subspace of $V$, i.e.\ it is also contained in a conjugate of $V\rtimes H$. 

The cases $\ell\in \{5,7\}$ can be dealt with by explicit Galois realizations from the databases. Alternatively, they are covered by the results of the next section.
Together with Theorem \ref{agl2cov}, we therefore obtain:
\begin{satz}
\label{agl2pol}
For every prime $\ell \ge 5$, there is an intersective polynomial with two monic irreducible factors $f,g\in \zz[X]$ such that $Gal(fg|\qq) = AGL_2(\ell)$.
\end{satz}

\subsection{Further affine groups}

We will generalize the results of the previous section for further affine groups, {\textit{conditional}} on the existence of certain tamely ramified Galois extensions of $\qq$.
More precisely we will show:
\begin{satz} 
\label{2traff}
Let $V=(\mathbb{F}_q)^n$ a finite vector space of dimension $n\ge 2$. Let $G=V\rtimes U$ be a finite affine group on $V$, with a group $U\le GL(V)$ which acts transitively on the 1-dimensional subspaces of $V$.\\
If $U$ occurs as a Galois group of a tamely ramified Galois extension over $\qq$, then there is an intersective polynomial $f\cdot g\in \zz[X]$ with two irreducible factors and with Galois group $G$.
\end{satz}

The first important observation for a proof of Thm.\ \ref{2traff} is that the $2$-covering constructed for the full affine groups $AGL_n(q)$ in the previous section can be generalized in a natural way to affine groups $G=V \rtimes U$ (with arbitrary $n\ge 2$, $V:=(\mathbb{F}_q)^n$ and $U\le GL_n(\mathbb{F}_q)$ transitive on the 1-dimensional subspaces). The proof is completely analogous to the one of Thm.\ \ref{agl2cov} (note that the transitivity of $U$ is needed because in the proof of Thm.\ \ref{agl2cov}, the corresponding stabilizers in $U$ need to be conjugate).

Now we will show that this $2$-covering has even stronger properties:
\begin{lemma}
\label{metacyc}
Let $V=(\mathbb{F}_q)^n$ a finite vector space of dimension $n\ge 2$. Let $G=V\rtimes U$ be as in Theorem \ref{2traff}.
Let $U_1$ be the stabilizer in $G$ of a point in $V$, and let $U_2:=V\rtimes H$, where $H \le U$ is a stabilizer of a one-dimensional subspace in $V$.
Then every metacyclic subgroup of $G$ is contained in a conjugate of $U_1$ or of $U_2$.
\end{lemma}
\begin{proof}
Consider the metacyclic subgroup of $G$ generated by two elements $\phi$ and $\psi$, with $\phi\psi = \psi \phi^k$ (for some $k\in \nn$). Let $\phi$ and $\psi$ act on $V$ via $\phi x = Ax+b$ and $\psi x = Cx+d$ (with $A,C\in U$ and $b,d\in V$). We have to show that either $\phi$ and $\psi$ possess a common fixed point (and hence $\langle \phi,\psi\rangle$ lies in a conjugate of $U_1$) or $A$ and $C$ possess a common eigenvector (and hence $\langle \phi,\psi\rangle$ lies in a conjugate of $U_2$).\\
From $\phi\psi = \psi \phi^k$, one computes easily that 
\begin{equation}
\label{matrix_equ1}
AC = CA^k \text{ and } b+Ad = d+C(1+A+\dots + A^{k-1})b
\end{equation}
{\textbf{Case 1:}} Assume that $\phi$ does not fix a point in $V$. In particular this implies $b\notin (A-1)V$, whence $A-1$ is not surjective. Write the characteristic polynomial of $A$ as $\chi_A(X)=(X-1)^m\dot f(X)$, with $m\ge 1$ and $f(X)$ coprime to $X-1$. Since $1$ can be written as a linear combination $u(X)(X-1)^m + v(X)f(X)$, the vector space $V$ decomposes as 
\begin{equation}
\label{v_decomp}
V=W\oplus E,
\end{equation}
where $E:=ker(A-1)^m$ and $W:=ker f(A)$.\\
Of course, $E$ and $W$ are both $A$-invariant. But they are also $C$-invariant. Indeed, from 
$(A-1)^mC = C(A^k-1)^m = C(1+A+\dots +A^{k-1})^m(A-1)^m$, it follows that $CE=E$. On the other hand, the decomposition in (\ref{v_decomp}) implies $(A-1)^mV = W$. Let $l\in \nn$ be such that $A^{kl}=A$ (such $l$ exists since conjugation by $C$ is an automorphism of $U$, whence $A^k$ generates $\langle A\rangle$). It follows that $A^l = CAC^{-1}$, and therefore $C(A-1)^m = (A^l-1)^mC = (A-1)^m(1+A+\dots + A^{l-1})^mC$. This shows that $W$ is $C$-invariant as well.

Next, we claim that one can assume $b\in E$ without loss. Indeed, conjugating $\phi$ by a translation $\tau\in G$, defined by $\tau x = x+a$ (with $a\in V$) yields
$$\tau\phi\tau^{-1} x = \tau\phi(x-a) = \tau(A(x-a)+b) = Ax-Aa+b+a,$$
which shows that we can subtract from $b$ any vector in $(A-1)V$, and therefore in particular any element of $(A-1)^mV=W$. Therefore, up to conjugating with a suitable translation, we can assume $b\in E$, as claimed. Equation (\ref{matrix_equ1}) then shows $(A-1)d\in E$, and so $d\in ker(A-1)^{m+1}=ker(A-1)^m = E$. Thus $\phi$ and $\psi$ restrict to affine linear maps on $E$. By slight abuse of notation, denote the restriction of $A$ and of $C$ to $E$ again by $A$ and $C$. In particular, $A$ is unipotent. Set $Q:=1+A+\dots + A^{k-1}$. It then holds that
$$(CQ-1)(A-1)=C(A^k-1)-(A-1)=(A-1)(C-1).$$
So $CQ-1$ maps $(A-1)E$ into itself. Assume that $CQ-1$ were invertible. Then $(CQ-1)x\in (A-1)E$ is equivalent to $x\in (A-1)E$, for any $x\in E$. But Equation (\ref{matrix_equ1}) yields $(CQ-1)b=(A-1)d$, which then implies $b\in (A-1)E$, contradicting the general assumption of Case 1. Thus $CQ-1$ is not invertible, i.e.\ $CQ$ has an eigenvalue $1$. Let $u\in E\setminus\{0\}$ be a corresponding eigenvector, and let $r\in \nn$ be minimal such that $(A-1)^r u = 0$. In particular $(A-1)^{r-1}u\ne 0$, and $(A-1)^{r-1}Au = (A-1)^{r-1}u$. Iteratively, 
\begin{equation}
\label{iteration}
(A-1)^{r-1}A^i u = (A-1)^{r-1}u \text{ for all } i\in \nn.
\end{equation}
It follows that $(A-1)^{r-1}Qu = k(A-1)^{r-1}u$. Since $k$ is coprime to $ord(A)$ and the latter is a power of the characteristic of $\mathbb{F}_q$ (due to the fact that $A$ is unipotent), we have $k\ne 0$ in $\mathbb{F}_q$. It follows that $v:=(A-1)^{r-1}Qu\ne 0$. We claim that $v$ is a common eigenvector of $A$ and $C$, thus concluding Case 1. Indeed, we have
$$(A-1)v = (A-1)^rQu = Q(A-1)^ru = 0$$
and 
$$Cv = C(A-1)^{r-1}Qu = (A^l-1)^{r-1}CQu = (A^l-1)^{r-1}u = $$
$$= (A-1)^{r-1}(1+A+\dots+A^{l-1})^{r-1}u  \stackrel{(\ref{iteration})}{=} l^{r-1}(A-1)^{r-1}u = (\frac{l}{k})^{r-1}v,$$
proving the claim.

{\textbf{Case 2}}: Assume now that $\phi$ fixes a point $u\in V$. Via conjugation by a translation, we may assume without loss that $u=0$. Therefore $\phi x = Ax$, and Equation (\ref{matrix_equ1}) reduces to
$$AC=CA^k \text{ and } (A-1)d=0.$$
Let $E_1:=ker(A-1)$. Since $d\in E_1$ and $E_1$ is $C$-invariant (compare the argument for $C$-invariance of $E$ in Case 1), $\psi$ maps $E_1$ into itself. If $\psi$ fixes a point in $E_1$, this point is a common fixed point of $\psi$ and $\phi$. Otherwise, the restriction of $C$ to $E_1$ has an eigenvalue $1$, and therefore $A$ and $C$ have a common eigenvector. This concludes the proof.
\end{proof}

We are now ready to prove Theorem \ref{2traff}.
In the proof, we will use results on solvability of split embedding problems with nilpotent kernel with extra conditions on ramification, presented in \cite[Thm.\ 9.5.11]{NSW} as part of the proof of Shafarevich's theorem.
This approach was also used by Sonn in \cite{Sonn08} to show the existence of intersective polynomials for solvable groups and Frobenius groups.
\begin{proof}[Proof of Thm.\ \ref{2traff}]
Let $n\ge 2$, let $G=V\rtimes U$ and
let $K|\qq$ be a tamely ramified Galois extension with group $U$. By \cite[Theorem IV.2.4]{MM}, there is an extension $L|K$ with $Gal(L|\qq)= G$ (split embedding problem with abelian kernel). 
Furthermore, by \cite[Thm.\ 9.5.11]{NSW}, $L$ can be chosen such that every ramified prime of $K|\qq$
splits completely in $L|K$, and every prime of $K$ that ramifies in $L$ is completely split in $K|\qq$. Therefore all decomposition groups of ramified primes in $L|\qq$ either lie in $V\subset U_2$ (for primes extending a completely split prime of $K|\qq$);
or they are metacyclic and intersect $V$ trivially (for primes that are completely split in $L|K$), in which case we have just shown in Lemma \ref{metacyc} that they lie in a conjugate of either $U_1$ or $U_2$ (here, $U_i$ are as in Lemma \ref{metacyc}).
This shows the assertion.
\end{proof}

{\textbf{Remark:}} The groups in Lemma \ref{metacyc} include in particular the doubly transitive affine groups. These groups are completely classified due to work of Hering (\cite{Hering}). They are of the form $G=V\rtimes U$ with $V=(\mathbb{F}_p)^n$ and one of the following:
\begin{itemize}
 \item[a)] $U\le \Gamma L_1(p^n)$,
 \item[b)] $U\trianglerighteq SL_a(q)$, with $q^a=p^n$,
 \item[c)] $U\trianglerighteq Sp_{2a}(q)$, with $q^{2a}=p^n$,
 \item[d)] $U\trianglerighteq G_2'(q)$, with $p^n=2^n=q^6$,
 \item[e)] $(V,U)$ in an explicitly known finite list of ``sporadic" cases. 
\end{itemize}
The groups in a) are solvable and have therefore already been dealt with in \cite{Sonn08}. The groups in b) include the case $U=GL_2(\ell)$ for which tamely ramified extensions are known (cf.\ Section \ref{pgl}); we therefore regain
the statement of Theorem \ref{agl2pol} as a special case of Theorem \ref{2traff}. However, the proof in Section \ref{agl2} still has its own merit as it works for {\textit{any}} solution of the $AGL_2(\ell)$-embedding problem beginning 
from the given $GL_2(\ell)$-extension, whereas Theorem \ref{2traff} requires very specific solutions for this embedding problem.

The groups in c) include in particular the groups of symplectic similitudes $GSp_4(\ell)$ for primes $\ell$. For these groups, the existence of tamely ramified Galois extensions of $\qq$ is known, again by work of Arias-de-Reyna and Vila (\cite{AdR2}).
This immediately yields:
\begin{kor}
Let $\ell$ be a prime. Then the affine symplectic group $AGSp_4(\ell)=(\mathbb{F}_\ell)^4 \rtimes GSp_4(\ell)$ occurs as the Galois group of an intersective polynomial $f\cdot g\in\zz[X]$ with exactly two irreducible factors. 
\end{kor}

\section{New explicit intersective polynomials for non-solvable groups}
\label{pols}
We present explicit infinite families of intersective polynomials with two irreducible factors for several non-solvable groups. 
Note that all of our families have a positive density subset of integer specializations of the extra parameter(s) yielding intersective polynomials. On the other hand, by Hilbert's irreducibility theorem,
the set of integer specializations that do not preserve the Galois group has density $0$. Therefore, one indeed obtains infinitely many intersective polynomials with the desired Galois group.

In Section \ref{prereq}, we give some prerequisites for the determination of decomposition subgroups which will be used in the following sections. Especially Section \ref{pgl28} uses some terminology from geometric Galois theory. For more on this, we refer to e.g.\ \cite{Voe}, and especially to the nice exposition in \cite[Section 2]{GMS} for background on monodromy groups of rational functions.
\subsection{Prerequisites for the computations}
\label{prereq}
To prove that a given polynomial is actually intersective, the following well-known lemma is useful. It connects the splitting behaviour of the mod-$p$ reduction of a polynomial with the orbits of the inertia and decomposition subgroups at primes
extending $p$.
\begin{lemma}
\label{orbits}
Let $f\in \zz[X]$ be monic and irreducible, let $p$ be a prime and $D_p$ resp.\ $I_p$ be the decomposition group resp.\ the inertia group of a prime extending $p$ in a splitting field of $f$, viewed as a permutation group on the zeroes of $f$.
Assume that $f$ splits as $f\equiv \prod_{i=1}^r f_i^{k_i}$ modulo $p$ (with $f_1,...,f_r$ irreducible). Then the orbit lengths of $D_p$ form a refinement of the partition $[n_1\cdot k_1,...,n_r\cdot k_r]$ of $n:=\deg(f)$, where $n_i:=\deg(f_i)$.
Similarly, the orbit lengths of $I_p$ form a refinement of
 $[\underbrace{k_1,...,k_1}_{n_1 \text{ times }},...,k_r,...,k_r]$.
\end{lemma}
\begin{proof}
The orbit lengths of $D_p$ are exactly the degrees of the irreducible factors of $f$ over $\qq_p$. Let $g$ be such a factor. By Hensel's Lemma (cf.\ e.g.\ Lemma II.4.6 in \cite{Neukirch}), 
the mod-$p$ reduction of $g$ can only have one irreducible factor
(up to multiplicity), so $g \equiv f_i^{m_i}$ mod $p$, for some $i\in \{1,...,r\}$ and $1\le m_i\le k_i$. This shows the assertion about $D_p$. The claim about $I_p$ follows analogously after replacing $\qq_p$ by its maximal unramified extension.
\end{proof}
This lemma will help to exclude certain ramification patterns for polynomials in an explicitly given family.

Another useful tool to investigate the roots of a polynomial over the $p$-adic field $\qq_p$ is the Newton polygon.

\begin{defi}[Newton polygon]
\label{newton}
Let $p$ be a prime, $f=\sum_{i=0}^n a_i x^i \in \qq_p[X]$ be a polynomial of degree $n$, with $f(0)\ne 0$, and let $\nu_p$ be the $p$-adic valuation on $\qq_p$.
The Newton polygon of $f$ is the set of line segments bounding the lower convex envelope of the point set $\{(i, \nu_p(a_i)): i\in \{0,...,n\}, a_i \ne 0\} \subset \rr^2$.
\end{defi}

Number the line segments of the Newton polygon of $f$ as $S_1,...,S_k$, where the slope of $S_{i+1}$ is strictly larger than the slope of $S_i$ (for all $i=1,...,k-1$). Then the lengths and slopes of the $S_i$ yield
information about the $p$-adic valuation of the roots of $f$. More precisely, the following holds (Proposition II.6.3 in \cite{Neukirch}):
\begin{prop}
\label{newt_prop}
With the above nomenclature, let $S$ be a line segment of the Newton polygon of $f$, beginning in $(r,\nu_p(a_r))$ and ending in $(s,\nu_p(a_s))$. Denote the slope of $S$ by $-m$. Then $f$ has precisely $s-r$ roots of $p$-adic valuation
$m$.
\end{prop}

Since the $p$-adic valuation is invariant for all roots of a given irreducible factor of $f$, Proposition \ref{newt_prop} immediately yields information about the Galois group of $f$ over $\qq_p$: its orbit lengths in the action
on the roots of $f$ must be a refinement of the partition given by the numbers $s-r$ (running through all line segments $S$ of the Newton polygon).

\subsection{The group $PSL_3(2)$}

We give a four-parameter polynomial $f$ leading to intersective integer polynomials $f\cdot g$ with two irreducible factors.
As the ``partner" polynomial $g$ would be rather complicated to write down with no parameters specialized, we only give a one-parameter version of it explicitly.
The polynomial $f$ first appeared in \cite{Ma1}.

Computations in this and the following sections were performed with the help of Magma. To ease verifiability, Magma code for some non-trivial calculations is available at\\ \texttt{mathematik.uni-wuerzburg.de/$\sim$koenig}.
\begin{satz}
Let $a,b,c$ and $t$ be algebraically independent transcendentals over $\qq$. 
Let $f:=X^7 - ((c - 2)a + 2b + c)X^6 + (-(b - 4)(c - 1)a^2 + ((c - 2)b^2
+(2c^2 - 5c + 4)b - 2c^2)a + b(2bc + 2c^2 + b^2 ))X^4
+((2c^2 - 1)(b - 4)a^2 + ((-2c^2 + c + 2)b^2 + (5c^2 + 2c - 4)b - 4c^2)a
-(c + 1)b^3 - c(2c + 3)b^2 + c^2b)X^3 + ((c^2 + 3c - 1)(4 - b) a^2
+((3c - 2)b^2 - 2(c^2 + 4c - 2)b + 4c^2 )a + b(b^2 + 3bc - c^2 ))cX^2
+(2abc - 8ac + ab - 4a - b^2 + 2bc)ac^2X - a^2(b - 4)c^3
+tX^2(X - c)(X^2 - bX + b)$.

Then $f$ has Galois group $PSL_3(2)$ over $\qq(a,b,c,t)$. Furthermore, for all odd integer specializations of $a,b,c$ and $t$ preserving the Galois group, there is an irreducible polynomial $g\in \zz[X]$ 
such that the specialized polynomial $f\cdot g$ has Galois group $PSL_3(2)$ over $\qq$ 
and has a root in every $\qq_p$.

In particular, for $a=b=c=1$ and $t\in \zz$ odd, the polynomial 
$g:=X^8 + (-14t + 14)X^7 + (87t^2 - 178t + 109)X^6 + (-314t^3 + 980t^2 - 1210t + 431)X^5 + (721t^4 - 3032t^3 + 5654t^2 - 4119t + 1293)X^4 + (-1080t^5 + 5700t^4 - 
    14238t^3 + 15835t^2 - 9776t + 2542)X^3 + (1032t^6 - 6521t^5 + 20397t^4 - 30663t^3 + 28012t^2 - 14338t + 3343)X^2 + (-576t^7 + 4212t^6 - 15786t^5 + 29976t^4 - 
    36081t^3 + 27147t^2 - 11847t + 1738)X + 144t^8 - 1188t^7 + 5169t^6 - 11874t^5 + 17689t^4 - 17396t^3 + 10799t^2 - 2948t + 295\in \zz[X]$
    fulfills the above requirements.
\end{satz}
\begin{proof}
The polynomial $f$ already has Galois group $G:=PSL_3(2)$ (in its natural degree 7 action) over $\qq(a,b,c,t)$ as shown in \cite{Ma1}, Theorem 4.3. A root field of $f$ inside a splitting field is therefore just the fixed field of a point stabilizer $U_1$ in this degree 7 action.

The polynomial $g$ was computed such that a root field of $g$ is the fixed field inside the splitting field of $f$
of a maximal subgroup $U_2$ of index $8$ in $G$ ($U_2$ is a point stabilizer of $G$, reinterpreted as $PSL_2(7)$).
Now the only maximal subgroups of $G$ not conjugate to $U_1$ or $U_2$ are the stabilizers of a two-dimensional subspace in the action on $(\mathbb{F}_3)^2$. These are isomorphic to $S_4$ (and conjugate to $U_1$ under an outer automorphism), and since e.g.\ their subgroups $S_3 \cong GL_2(2)$ also fix a one-dimensional subspace, one easily concludes that the only proper subgroups of $G$ not contained in a conjugate of $U_1$ or of $U_2$ are isomorphic to $A_4$ or $S_4$.\footnote{See also \cite[Satz 8.27]{Hu1} for a general classification of subgroups of $PSL_2(q)$, originally due to Dickson.}

This means that all specializations of $f\cdot g$ that preserve the Galois group have a root in every $\qq_p$ such that $p$ is (at most) tamely ramified in the splitting field over $\qq$; and in fact in {\textit{all}} $\qq_p$, as long as $2$ is not 
wildly ramified.
Now for $a,b,c,t$ all odd, $f$ is congruent modulo 2 to $X^7 + X^5 + X^4 + X^3 + 1$ which is an irreducible separable polynomial over $\mathbb{F}_2$. This means that the splitting field of $f$ over $\qq_2$ is unramified (of degree $7$), which completes
the proof.
\end{proof}

\subsection{The group $P\Gamma L_2(8)$}
\label{pgl28}
From Theorem \ref{pgl2cov}, the group $PSL_2(8)=PGL_2(8)$ has a $2$-covering by subgroups $U_1$ of index $9$ (point stabilizer in the natural action) and $U_2$ of index $28=\frac{8\cdot 7}{2}$. The analagous statement remains true for $P\Gamma L_2(8)=PSL_2(8).3$, with the subgroups $U_1.3$ and $U_2.3$. In fact, for both $PSL_2(8)$ and $P\Gamma L_2(8)$, most solvable subgroups are already contained in a conjugate of one of the two groups used for the $2$-covering:
the only exceptions are of the form $C_7 \rtimes V$ with $V\le Aut(C_7)$.
For $PSL_2(8)$, this is because the only maximal subgroups not conjugate to $U_1$ or $U_2$ are isomorphic to $D_7$; for $P\Gamma L_2(8)$, the only maximal subgroups not conjugate to $U_1.3$ or $U_2.3$ are $AGL_1(7)$ (all of whose non-cyclic subgroups are of the form $C_7 \rtimes V$ as above) and $PSL_2(8)$ (whose subgroups are already dealt with by the above).

This means that any Galois extension of $\qq$ with group $PSL_2(8)$ or $P\Gamma L_2(8)$ will yield an intersective polynomial with two irreducible factors, just as long as we can make sure that no inertia subgroup $I_p$ has order divisible by $7$
(for any prime $p$). We will restrict to the case of $P\Gamma L_2(8)$ here. 

There are several known polynomials with Galois group $P\Gamma L_2(8)$ over $\qq$ (see e.g. the tables in the appendix of \cite{MM}), and those can indeed be used to find intersective polynomials.
In the following, we use a family of polynomials with nice additional properties to obtain infinitely many intersective polynomials for $P\Gamma L_2(8)$. Since I did not find this specific family in the literature, a strict proof for the Galois group is included below.

\begin{satz}
Let $t$ be transcendental over $\qq$,
$f:=(x^3 + 16x^2 + 160x + 384)^3- 7^3\cdot t(x^2+13x+128)$ and\\
$g:=(x^9 + 11x^8 + 4x^7 - 868x^6 + 6174x^5 - 43974x^4 + 37492x^3 - 28852x^2 - 2967x + 211)^3(x-5)+2^7\cdot 7^3\cdot t(x^3 - x^2 - 9x + 1)^7$.
Then $f\cdot g$ has Galois group $P\Gamma L_2(8)$ over $\qq(t)$, and for every odd integer specialization $t\mapsto t_0\in 2\zz+1$ preserving the Galois group, the specialized polynomial $f\cdot g$ has a root in every $\qq_p$.
\end{satz}

\begin{proof}
We first show $Gal(f\cdot g|\qq(t))=P\Gamma L_2(8)$.
Let $L|\qq(t)$ be a splitting field of $f(t,X)$ and $G:=Gal(L|\qq(t))$. 

Define $f_1,f_2\in \qq[x]$ by $f(t,x)=f_1(x)-tf_2(x)$. Since the rational function $t(x):=f_1(x)/f_2(x)$ has a pole of order $7$ and two simple poles, $G$ contains a $7$-cycle (namely an inertia group generator at $t\mapsto \infty$). Of course $G$ then acts primitively. A well-known theorem usually attributed to Marggraff (e.g.\ \cite[Exercise 7.4.11]{DM}) implies that $G$ is $3$-fold transitive, and therefore one of $PSL_2(8)$, $P\Gamma L_2(8)$, $A_9$ and $S_9$. Since $t\mapsto \infty$ is a rational branch point of $L|\qq(t)$, Fried's branch cycle lemma (e.g.\ \cite[Lemma 2.8]{Voe}) implies that the corresponding inertia group generator has to belong to a rational conjugacy class of $G$ (i.e.\ it has to be conjugate in $G$ to all its non-trivial powers). This implies $G\ne PSL_2(8)$, where the normalizer of a subgroup $C_7$ is only $D_7$.

To conclude $G\notin\{A_9,S_9\}$, we will detect a certain intransitive subgroup of $G$ by constructing some rational function field inside $L$ over which $f$ becomes reducible.\footnote{This approach is common for Galois group verification over $\qq(t)$; see e.g.\ \cite{Ma1} for further examples.}
Let $y$ be a root of $g$ and let $z$ be such that $y=\frac{z^3+z+2}{z^2+6z+37}$.
In particular $\qq(z)|\qq(t)$ is an extension of rational function fields, so $t=\frac{h_1(z)}{h_2(z)}$ with polynomials $h_1,h_2\in \qq[x]$ (which can of course be computed explicitly from the above data, although they are rather large, whence we do not include them here).
Now consider the polynomial $f_1(x)h_2(z)-f_2(x)h_1(z) \in \qq[x,z]$.
Computation with Magma shows that this polynomial is reducible, with factors of degree $3$ and $6$ in $x$. This shows that $f(t,x) = f_1(x)-tf_2(x)$ becomes reducible over $\qq(z)$, which is a degree-$84$ extension of $\qq(t)$ containing the degree-$28$ subextension $\qq(y)|\qq(t)$. But then $f$ is also reducible with the same factorization over $\qq(z)\cap L$. Now $(\qq(z)\cap L)|\qq(t)$ is either of degree strictly smaller than 84 or possesses a non-trivial intermediate field (of degree $28$). This means that $G$ possesses a subgroup $U$ with orbit lengths $3$ and $6$, such that either $[G:U]<84$ or $U$ is not maximal in $G$. Since the $3$-set stabilizers in $A_9$ and in $S_9$ are maximal of index $(9\cdot 8\cdot 7)/6=84$, it follows that $G=P\Gamma L_2(8)$.
Since this group is $3$-fold transitive, all $3$-set stabilizers are of index $(9\cdot 8\cdot 7)/6=84$, which implies $\qq(z)\cap L = \qq(z)$. In particular, $\qq(y)$ is contained in $L$, and so the splitting field of $g$ is $L$ as well. Altogether $Gal(f\cdot g|\qq(t))=P\Gamma L_2(8)$, as claimed.

Now we show the claim about intersective polynomials. Let $p$ be a prime, $t_0\in \zz$, and denote the specialization of $f$, corresponding to $t\mapsto t_0$, by $f_0$.
By the arguments preceding the theorem, we only have to exclude that the inertia group $I_p$ of $f_0$ at $p$ contains a group of order $7$. If it does, then $f_0$ decomposes 
in the form $h_0:=(x-a)^7\cdot (x^2+bx+c)$ over $\mathbb{F}_p$ (with some $a,b,c\in \mathbb{F}_p$), since $I_p$ has an orbit of length 7 in the action on 9 points.
We solve the equation $f_0\equiv h_0$ mod $p$ by computing (with Magma) a Groebner basis of the ideal generated by the coefficients of $f-h$ in $\zz[a,b,c,t]$ (where $h$ is defined like $h_0$, but with the coefficients viewed as indeterminates).
This basis contains an element $2^{11}\cdot 7$. Therefore, if $c_0(a,b,c,t),...,c_8(a,b,c,t) \in \zz[a,b,c,t]$ are the coefficients of $f-h$, there exist polynomials $m_0,...,m_8\in \zz[a,b,c,t]$ with
$\sum_{i=0}^8 m_i c_i(a,b,c,t) = 2^{11}\cdot 7$.
The same equality then holds of course after any integer evaluation of the variables $a,b,c,t$. As $f_0\equiv h_0$ mod $p$ forces all the $c_i$ (evaluated at some integral $a,b,c,t$) to be divisible by $p$,
we obtain $2^{11}\cdot 7 \equiv 0$ mod $p$, and therefore $p\in \{2,7\}$. 

Now for odd $t$, $f$ is separable modulo 2 and therefore unramified at 2.

For $p=7$, consider the Newton polygon corresponding to the polynomial $g(x+5)$ over $\qq_7$. The points $(i,\nu_7(a_i))$ determining this polygon are $(0,k)$ with $k\ge 10$ (exact value depending on $t_0$), $(1,9)$ and $(28,0)$.
All other points lie on or above the line connecting $(1,9)$ and $(28,0)$. By proposition \ref{newt_prop}, this means that $g(x+5)$ has exactly one root of $7$-adic valuation $\ge 1$, and therefore the decomposition groups at 7 fix a root of $g$,
for all specializations $t\mapsto t_0\in \zz$. This concludes the proof.
\end{proof}

{\textbf{Remark:}}\\
The above family of intersective polynomials has the nice additional property that both $f$ and $g$ are of the form $f(t,x)=p_1(x)-tq_1(x)$ and $g(t,x)=p_2(x)-tq_2(x)$, with polynomials $p_1,q_1,p_2,q_2$. 
After replacing $t$ by $2t+1$, the above proof shows that in fact $f\cdot g$ has a root for {\textit{all}} evaluations of $t$ to integer values (although not always with $Gal(f\cdot g)=P\Gamma L_2(8)$).
In other words, there are rational functions $r_1(x)$ and $r_2(x) \in \qq(x)$ (namely $r_i=p_i/q_i$) such that for all primes $p$ and all integer values $t_0\in \zz$, at least one of the rational functions $r_1(x)$ and $r_2(x)$ takes the value $t_0$ in $\qq_p$. Even stronger, let $t_1\in \zz_p$ be a $p$-adic integer. Krasner's lemma yields that, if $t_0\in \zz$ is an integer sufficiently close ($p$-adically) to $t_1$, then $f(t_0,x)$ has a root in $\qq_p$ if and only if $f(t_1,x)$ does.\footnote{To be precise, $f(t_1,x)$ has to be assumed separable here; this however excludes only finitely many values for $t_1$, and in the example of the above theorem, these can be treated separately.}
Therefore, the previous can be generalized to all $p$-adic integer values: For all primes $p$ and all $t_1\in \zz_p$, at least one of the rational functions $r_1$ and $r_2$ takes the value $t_1$ in $\qq_p$.\\
This property is itself worth investigating in generality. Namely, assume that $E|\qq(t)$ is the splitting field of $(r_1(x)-t)\cdot (r_2(x)-t)$ over $\qq(t)$ and $G=Gal(E|\qq(t))$.
Denote by $U_1, U_2\le G$ the Galois groups of $E|K_i$, where $K_i$ is a root field of $r_i-t$ ($i=1,2$).
Let $(\sigma_1,...,\sigma_r)$ be the {\textit{branch cycle description}} of $E|\qq(t)$ (with $\sigma_i\in G$, $i=1,...,r$). Then the following two conditions hold:
\begin{itemize}
\item[i)] $(\sigma_1,...,\sigma_r)$ is a {\textit{genus zero system}} in both permutation actions on $G/U_1$ and on $G/U_2$ (since the fixed fields of $U_i$ are of genus zero).
\item[ii)] The subgroups $U_1$ and $U_2$ yield a $2$-covering of $G$ (because of the intersective property).
\end{itemize}
Since the combination of these two conditions is quite restrictive, a complete classification of such pairs of rational functions might be achieved.

\subsection{The group $M_{11}$}
The smallest Mathieu group $M_{11}$ is $2$-coverable in different ways. In order to give intersective polynomials of small degree, we use a $2$-covering by the two maximal subgroup of smallest index, namely
the point stabilizers in the $3$-transitive actions on 11 and 12 points respectively. One checks directly with a computer that the only solvable subgroups of $M_{11}$ not contained in the union of conjugates of these two subgroups
are metacyclic groups $C_3.C_6$ of order $18$ and several non-metacyclic groups, all of order dividing 144.

The polynomials used in the following theorem were first found by the author in \cite{Koe} (Theorem 1). 
For the sake of simplicity, we give rational, but not integer polynomials $g$ in the theorem below. Turning these into monic integer polynomials is trivial, but would blow up the coefficients unnecessarily.
\begin{satz}
Let $f:=(x^2-4x-16)^5(x^2-4x-1)+2tx^3(x-4)^3(x^3-128)$ and
$g:=x^{11}+x^7(3x+2)s+x(3x^4 + 14/5x^3 + 4/5x^2 - 40/81x - 16/81)s^2-(x-2/5)(x+2/5)s^3$, where $s:=4t/10125$.

Then $fg$ has regular Galois group $M_{11}$ over $\qq(t)$, and the specialized polynomial is intersective for all specializations $t\mapsto t_0\in \zz$ that preserve the Galois group and fulfill $t_0\equiv 1$ mod $3$.
\end{satz}
\begin{proof}
Once again, the polynomial $g$ was computed to parameterize the stem field of an index-11 subgroup in $Gal(f|\qq)\cong M_{11}$.
By the argument preceding this theorem, in order to obtain intersective polynomials from a given $M_{11}$-extension, we only have to exclude wild ramification at $2$ and $3$, and ramification index 3 at any prime.
Since the subgroups of order $3$ in $M_{11}$ have exactly three orbits of length 3, ramification index 3 at a prime $p$ would mean that $f\equiv (x^3+a_1x^2+a_2x+a_3)^3(x^3+a_4x^2+a_5x+a_6)$ mod $p$.
Groebner basis computation as in the previous section 
shows that this is only possible for $p\in \{2,3,5,11\}$. 

Modulo 11, $f$ has a simple zero at $x=6$ for all integer specializations of $t$, which means that the decomposition group at $11$ is contained in a point stabilizer of the degree-12 action.

Modulo 2, $f$ factors as $x^{10}\cdot (x-1)^2$, which means that the decomposition group at $2$ has an orbit of length at most 2 in its action on 12 points. 
However, one checks directly that the orbit lengths of the ``bad" candidate subgroups mentioned
before the theorem are $(12)$, $(8,4)$, $(6,6)$ and $(6,3,3)$ respectively. Therefore the decomposition group at $2$ is not one of them.

To gain information modulo 3, set $t=1$. The polynomial $3^{12}\cdot g$ is then congruent to $2x+2$ mod $3$. Therefore its Newton polygon has a length-1 line segment of slope $0$; in other words, $g$ has a linear factor over $\qq_3$.

Finally, modulo 5, $f$ has at least one simple zero as long as $t$ is not divisible by $5$. Therefore, in these cases, the decomposition groups at $5$ fix a point.\\
If $5|t$, but $5^6$ does not divide $t$, then the Newton polygon for $f(x-3)$ shows that $f$ still has a linear factor over $\qq_5$ 
(the $5$-adic valuations of the coefficients $a_i$ of this polynomial are $\nu_5(a_0)\ge1+\nu_5(t)$, $\nu_5(a_1)=\nu_5(t)$, 
and all other valuations $\nu_5(a_i)$ strictly larger than $\nu_5(t)-(i-1)$. Therefore the Newton polygon contains a length-1 line
segment of slope $-1$). On the other hand, if $5^6|t$, then an analogous argument with the Newton polygon of $g$ shows that $g$ has a linear factor over $\qq_5$.
This concludes the proof.
\end{proof}

\end{document}